\documentclass[12pt,reqno,a4paper]{amsart}
\usepackage{amsmath,amssymb,amsthm,amsaddr}
\usepackage{enumitem}
\usepackage{mathtools}
\usepackage[margin=3cm,top=3.5cm,bottom=3.5cm,footskip=1cm,headsep=1.5cm]{geometry}
\usepackage[utf8]{inputenc}

\usepackage[british]{babel}

\usepackage{color} 
\usepackage{hyperref}
\hypersetup{
    colorlinks,
    citecolor=red,
    filecolor=black,
    linkcolor=blue,
    urlcolor=black
}

\usepackage{mathtools}
\usepackage{float}
\usepackage{tikz}
\usetikzlibrary{calc}
\usetikzlibrary{matrix}

\usepackage{color}
\usepackage[caption=false]{subfig}

\author{H. Egger \and B. Radu}
\address{Department of Mathematics, TU Darmstadt, Germany}
\email{egger@mathematik.tu-darmstadt.de}
\email{radu@gsc.tu-darmstadt.de}

\definecolor{mygray}{rgb}{.5,.5,.5}

\title[Second order mass-lumping for Maxwell's equations]{A second order finite element method with mass\\lumping for Maxwell's equations on tetrahedra}

\newtheorem{lemma}{Lemma}[section]
\newtheorem{problem}[lemma]{Problem}
\newtheorem{theorem}[lemma]{Theorem}

\theoremstyle{definition}
\newtheorem{remark}[lemma]{Remark}

\def\dt{\partial_t}
\def\dtt{\partial_{tt}}
\def\dttt{\partial_{ttt}}
\def\dtttt{\partial_{tttt}}
\def\dtau{d_\tau}
\def\dtautau{d_{\tau\tau}}

\def\div{\mathrm{div}}
\def\curl{\mathrm{curl}}
\def\N{\mathcal{N}}
\def\EJ{\mathcal{E}\!\mathcal{J}}
\def\mEJ{\EJ^*}
\def\BDM{\mathcal{BDM}}

\def\Th{\mathcal{T}_h}
\def\K{K}
\def\tnorm{|\!|\!|}
\def\E{\mathcal{E}}

\def\RR{\mathbb{R}}
\newcommand\numberthis{\addtocounter{equation}{1}\tag{\theequation}}

\begin{document}

\begin{abstract}
We consider the numerical approximation of Maxwell's equations in time domain by a second order $H(\curl)$ conforming finite element approximation. In order to enable the efficient application of explicit time stepping schemes, we utilize a mass-lumping strategy resulting from numerical integration in conjunction with the finite element spaces introduced in \cite{ElmkiesJoly97}. We prove that this method is second order accurate if the true solution is divergence free, but the order of accuracy reduces to one in the general case. We then propose a modification of the finite element space, which yields second order accuracy in the general case.
\end{abstract}

\maketitle

\begin{quote}
\end{quote}


\section{Introduction}
We consider the efficient numerical solution of electromagnetic wave propagation modeled by time-dependent Maxwell's equations in second order form
\begin{align*}
\varepsilon \dtt E + \curl\,(\mu^{-1}\curl\,E) = -\dt j.
\end{align*}
Here $E$ is the electric field, $\varepsilon$ and $\mu$ are the symmetric and positive definite permittivity and permeability tensors, and $j$ is the current density. Conduction currents can be included by setting $j=j_{s} + \sigma E$, where $j_s$ are the impressed source currents and $\sigma$ is the electric conductivity.  

Today's industry standard for solving Maxwell's equations in time-domain are the finite difference time domain method and the finite integration technique \cite{Weiland77,Yee66} which provide second order approximations on rectangular grids and for smooth and isotropic coefficients . Due to the underlying explicit time-stepping schemes, they lead to very efficient and accurate numerical approximations. 
Non-trivial modifications are, however, required to guarantee stability of the schemes in the case of non-rectangular grids and discontinuous or anisotropic coefficients \cite{SchuhmannWeiland98}, which in general also leads to reduced convergence orders. 

A flexible alternative is provided by Galerkin approximations based on $H(\curl)$ conforming finite elements, for which a rigorous stability and convergence analysis is possible under rather general assumptions. The standard finite element approximation with N\'ed\'elec elements $\N_{k-1}$ of order $k$, leads to error estimates of the form 
\begin{align*}
\|\dt (E(t) - E_h(t))\|_{L^2(\Omega)} + \|\curl (E(t) - E_h(t))\|_{L^2(\Omega)} \le C(E) h^k, 
\end{align*}
which are optimal in view of the approximation properties of these spaces; see \cite{MakridakisMonk95,Monk92,Monk93} for details. 
A major drawback of standard finite element approximations for wave propagation problems, however, is that due to the required $H(\curl)$ conformity of the basis functions, the linear systems 
\begin{align*}
M\ddot e(t) + K e(t) = g(t)
\end{align*}
arising from discretization in space have a mass matrix $M$ which is sparse, but does not have a sparse inverse. This prohibits an efficient time integration by explicit time stepping schemes. 

In order to overcome this source of inefficiency, mass-lumping strategies can be applied, which aim at replacing the mass matrix $M$ by a (block) diagonal approximation $\widetilde M$, in such a way that the overall accuracy of the approximation is not reduced. A systematic analysis of such schemes is possible, if mass-lumping can be interpreted as inexact numerical integration. 
In this spirit, mass-lumping for finite element methods for Maxwell's equations on quadrilateral and hexahedral grids has been investigated in \cite{CohenMonk98}. In fact, a close relation exists between finite difference schemes \cite{Weiland77,Yee66} and low order finite element approximations with mass-lumping; we refer to \cite{Cohen02,Joly03} for details. 

In order to obtain the full geometric flexibility of finite element approximations, we here consider mass-lumping for Maxwell's equations on tetrahedral meshes, for which only few results are available.
Lowest order N\'ed\'elec elements of type one and two has been proposed in \cite{EggerRadu19} and first order convergence has been established. Related methods have been proposed in \cite{CodecasaPoliti08} in the context of the finite integration technique, but no convergence analysis is given there. A mass-lumping strategy based on an extension of the lowest order elements has been proposed by Elmkies and Joly \cite{ElmkiesJoly97} and first order convergence has been illustrated by a numerical dispersion analysis.
Second order convergence has been observed by the authors for an appropriate extension of the second order N\'ed\'elec element $\N_1$, which we call $\EJ_1$ element in the following; we refer to \cite{Cohen02,ElmkiesJoly97} for details. 
%

As a first result of this paper, we will prove that 
\begin{itemize}
\item the $\EJ_1$ element with mass-lumping yields in fact second order convergence, if $\div(\varepsilon E) = 0$, but in general, only first order convergence can be obtained. 
\end{itemize}
Since $\div(\varepsilon E)=0$ is satisfied when $\div\,j=0$, our analysis also explains the good convergence behavior observed in the numerical tests in \cite{Cohen02,ElmkiesJoly97}. 
Our proof of second order convergence when $\div(\varepsilon E)=0$ is based on a detailed analysis of quadrature errors, which also provides insight into the cause for the convergence order reduction in the general case. This allows us to 
\begin{itemize}
\item propose a modification $\mEJ_1$ of the $\EJ_1$ element which, together with appropriate mass-lumping, leads to second order convergence in the general case.
\end{itemize}
In fact, only one of the basis functions of the $\EJ_1$ element has to be slightly changed.
%
In summary, we obtain a second order inexact Galerkin approximation of Maxwell's equations with block diagonal mass matrix with the same accuracy and flexibility of standard finite element approximations. 

The focus of this manuscript lies on the second order approximations for Maxwell's equations, but the basic arguments can in principle also be used for the construction and analysis of mass-lumping schemes for other equations and approximations of higher order. Some ideas in these directions will be discussed at the end of the manuscript.
Let us note that some additional degrees of freedom are required for mass-lumping, whose number increases with higher order of approximation. We therefore expect that discontinuous Galerkin methods \cite{HesthavenWarburton08}, which also have (block) diagonal mass matrices, become more efficient for higher polynomial degree. A thorough comparison of finite elements with mass-lumping and discontinuous Galerkin schemes is given in \cite{GeeversMulderVegt18} in the context of elastodynamics. 

The remainder of the manuscript is organized as follows:
In Section~\ref{sec:prelim}, we briefly summarize some results about the discretization of electromagnetic wave propagation problems by inexact Galerkin methods in space and explicit time integration scheme.
A convergence analysis is given under some simple abstract conditions, which can 
easily be verified for particular approximations. As an example, in Section~\ref{sec:nedso}, we apply the results to the 
standard $\N_1$ element.
%
In Section~\ref{sec:ej1fo}, we then analyze the effect of inexact numerical integration for the $\EJ_1$ element.
Sections~\ref{sec:ej1so} and \ref{sec:42} then contain our main results: We first show that the $\EJ_1$ element leads to second order convergence, if $\div(\varepsilon E)=0$, and then propose and analyze the new $\mEJ_1$ element, which leads to second order convergence in the general case.
Some numerical tests are presented in Section~\ref{sec:num} for illustration of our theoretical results. 
In Section~\ref{sec:discussion}, we briefly review the main ingredients that are required to obtain higher order approximations or discretizations for other types of equations.
Detailed proofs for some technical lemmas and a list of basis functions for the $\mEJ_1$ element are provided in the appendix.

\section{Inexact Galerkin approximations}\label{sec:prelim}
We consider Maxwell's equations 
\begin{align} \label{eq:maxwell1}
\varepsilon \dtt E(t) + \curl(\mu^{-1} \curl\,E(t)) &= f(t) \qquad \text{on } \Omega, \ t>0 \end{align}
on some bounded Lipschitz domain $\Omega \subset \RR^3$. For ease of presentation, we complement \eqref{eq:maxwell1} by homogeneous boundary and initial conditions 
\begin{align}
n \times E(t) &= 0 \qquad \text{on } \partial\Omega, \ t>0, \label{eq:maxwell2} \\
E(0) = 0, \qquad \dt E(0) &= 0 \qquad \text{on } \Omega. \label{eq:maxwell3}
\end{align}
More general boundary and initial conditions and also lower order terms in \eqref{eq:maxwell1} can be considered with minor modifications. 
To avoid technicalities, we assume that 
\begin{align}
 &\text{$\varepsilon$, $\mu$ are symmetric positive definite matrices, and that}  \label{eq:ass1}\\
 &\text{$f$ is a smooth function of time with values in $L^2(\Omega)$}. \label{eq:ass2} 
\end{align}
Piecewise smooth coefficient functions and more general right hand sides can again be considered with minor modifications. 
Under assumptions \eqref{eq:ass1}--\eqref{eq:ass2}, the existence of a unique solution to \eqref{eq:maxwell1}--\eqref{eq:maxwell3} can be proven by semi-group theory~\cite{Leis85,Pazy83}, and solutions of \eqref{eq:maxwell1} are characterized by the variational principle 
\begin{align} \label{eq:var}
(\varepsilon \dtt E(t), v) + (\mu^{-1} \curl\,E(t), \curl\,v) &= (f(t),v),
\end{align}
for all $v \in H_0(\curl;\Omega)$ and $t>0$; see \cite{Monk92} for details.
Here and below, $(\cdot,\cdot)$ denotes the scalar product of $L^2(\Omega)$.

\subsection{Space discretization}

Let $\Th$ denote a shape-regular tetrahedral mesh of the domain $\Omega$. 
We denote by $H^k(\Th)$ the spaces of piecewise smooth functions over the mesh $\Th$, and write
\begin{align}\label{eq:nedelec}
\N_k(\Th) = \{ v : v|_{K} \in \N_k(K)  \text{ for all } K \in \Th\}
\end{align} 
for the space of piecewise polynomial functions whose restrictions to any element $T$ belong to the N\'ed\'elec space $\N_k(K)= P_k(\K)^3 \oplus [x \times P_k^h(\K)^3] $; see \cite{BoffiBrezziFortin13}. 
We look for approximations for $E(t)$ in a finite element space $V_h$ satisfying
\begin{itemize}
	\item[(A0)]  $\N_k(\Th)\cap H_0(\curl,\Omega)\subset V_h \subset H_0(\curl,\Omega)$.
\end{itemize}
There exist locally defined projection operators $\Pi_h^k : H^1(\Th)^3 \to \N_k(\Th)$ such that 
\begin{alignat}{2}
\|\Pi_h^k v - v\|_{L^2(\K)} &\le c h^{s} \|v\|_{H^s(\K)}, \qquad &1 \le s \le k+1,  \label{eq:projest1}\\
\|\curl (\Pi_h^k v - v)\|_{L^2(\K)} &\le c h^{s} \|\curl\,v\|_{H^s(\K)}, \qquad &1 \le s \le k+1. \label{eq:projest2}   
\end{alignat}
Due to shape-regularity of the mesh $\Th$ the constant $c$ can be chosen independent of the element $\K$. 
We further denote by $\pi_h^m : L^2(\Omega) \to P_m(\Th)$ the $L^2$-projection operator to piecewise polynomials of order $m$ and note that 
\begin{align*}
\|\pi_h^m v - v\|_{L^2(\K)} \le C h^{s} \|\nabla^s v\|_{L^2(\K)} \qquad 0 \le s \le m+1.
\end{align*}
The definition of the $L^2$-projection naturally extends to vector valued functions.

For the numerical approximation of problem \eqref{eq:maxwell1}--\eqref{eq:maxwell3}, we then consider inexact Galerkin finite element methods of the following form. 
\begin{problem} \label{prob:semi}
Find $E_h : [0,T] \to V_h$ such that $E_h(0)=\dt E_h(0)=0$ and 
\begin{align} \label{eq:varh}
(\varepsilon \dtt E_h(t), v_h)_h + (\mu^{-1} \curl\,E_h(t), \curl\,v_h) &= (f(t),v_h),
\end{align}
for all $v_h \in V_h$ and $t > 0$. Here $(\cdot,\cdot)_h$ denotes a suitable approximation for the scalar product $(\cdot,\cdot)$ on $L^2(\Omega)$, which is part of the definition of the method.
\end{problem}
\noindent
In order to ensure the well-posedness of the semi-discrete problem, we require that 
\begin{itemize}
\item[(A1)] the bilinear form $(\varepsilon \cdot,\cdot)_h$ defines a scalar product on $V_h$ and
\begin{align*}
c_1 (\varepsilon v_h,v_h) \le (\varepsilon v_h,v_h)_h \le c_2 (\varepsilon v_h,v_h) \qquad \forall v_h \in V_h,
\end{align*}
for some positive constants $c_1,c_2$. 
\end{itemize}
As a consequence of this assumption and the positivity of $\varepsilon$, we can estimate 
\begin{align}
(\varepsilon v_h,w_h)_h \le C \|v_h\|_{L^2(\Omega)} \|w_h\|_{L^2(\Omega)} \qquad \text{for all } v_h,w_h \in V_h.
\end{align}
By choosing any basis for the finite dimensional space $V_h$, we can transform the discrete variational equation \eqref{eq:varh} into a linear system
\begin{align} \label{eq:varh_ls}
\widetilde M \ddot e(t) + K e(t) &= g(t), \qquad t>0,  
\end{align}
describing the evolution of the coordinate vector $e(t)$ representing the finite element function $E_h(t)$. 
Due to condition (A1), the mass matrix $\widetilde M$ in \eqref{eq:varh_ls} is symmetric positive definite, and existence of a unique solution for given initial values $e(0)=\dot e(0)=0$ follows from the Picard-Lindelöf theorem. This also implies the well-posedness of Problem~\ref{prob:semi}.
As a second ingredient, we assume that 
\begin{itemize}
\item[(A2)] the inexact scalar product $(\cdot,\cdot)_h$ is sufficiently accurate, i.e. there exists a constant $c_\sigma\ge 0$ such that
\begin{align*}
|\sigma_h(\pi_h^{*} E,v_h)| \le c_\sigma h^{q+1} \|E\|_{H^{q}(\Th)} \|\curl\,v_h\|_{L^2(\Omega)},  
\end{align*}
where $\sigma_h(u_h,v_h) := (\varepsilon u_h,v_h) - (\varepsilon u_h,v_h)_h$ represents the quadrature error and $\pi_h^{*}:H^1(\Th)^3\to P_q(\Th)^3$ is a suitable projection operator satisfying
\begin{align*}
	\|\pi_h^* E - E\|_{L^2(\K)} &\le c h^{q+1} \|E\|_{H^{q+1}(\K)},\qquad \forall\K\in\Th.
\end{align*}
\end{itemize}
Based on these general assumptions, we can prove the following convergence result. 
\begin{theorem} \label{thm:1}
Let (A0)--(A2) hold and let $E$ denote a sufficiently smooth solution of \eqref{eq:maxwell1}--\eqref{eq:maxwell3}.
Then the inexact Galerkin approximation $E_h$ of Problem~\ref{prob:semi} satisfies 
\begin{align}
\|\dt (E - E_h)\|_{L^\infty(0,T,L^2(\Omega))} + \|\curl (E - E_h)\|_{L^\infty(0,T,L^2(\Omega))}
\le C(E) h^{r},
\end{align}
with rate $r=\min\{k+1,q+1\}$ and constant
\begin{align*}
C(E) &= C \, \Big(
\|\dt E\|_{L^\infty(0,T,H^{r}(\Th))} + \|\dtt E\|_{L^1(0,T,H^{r}(\Th))} \\
& \qquad + \|\curl\,E\|_{L^\infty(0,T,H^{r}(\Th))} + \|\curl\,\dt E\|_{L^1(0,T,H^{r}(\Th))}  \\
& \qquad +c_\sigma\|\dtt E\|_{L^\infty(0,T,H^{r-1}(\Th))} + c_\sigma\|\dttt E\|_{L^1(0,T,H^{r-1}(\Th))} \Big).
\end{align*}
and constant $C$ only depending on the domain, the shape-regularity of the mesh, and the bounds for the coefficients. 
\end{theorem}
The result follows from standard energy arguments and assumptions (A0)--(A2). 
For convenience of the reader, a complete derivation is given in the appendix. 

\subsection{Time discretization}
Let us also consider the time-discretization by a typical explicit scheme, which will be used in our numerical tests. 
\begin{problem}[Fully discrete scheme] \label{prob:full}
Set $E_h^0=E_h^1=0$, and then determine $E_h^{n}$ for $n \ge 1$ by the variational equations
\begin{align} \label{eq:varhtau}
(\varepsilon \dtautau E_h^n, v_h)_h + (\mu^{-1} \curl\,E_h^n, \curl\,v_h) &= (f(t^n),v_h) \qquad \forall v_h \in V_h.
\end{align}
\end{problem}
\noindent
Here $E_h^n$ is the approximation for the semi-discrete solution $E_h(t^n)$ at time $t^n=n \tau$ resulting from time discretization, and $\tau>0$ is the time step size. Furthermore, 
\begin{align}
\dtau E_h^{n+\frac{1}{2}} = \frac{1}{\tau}(E_h^{n+1}-E_h^{n})
\quad \text{and} \quad
\dtautau E_h^n = \frac{1}{\tau^2} (E_h^{n+1} -  2 E_h^n + E_h^{n-1})
\end{align}
are the usual central difference quotients of first and second order. Moreover, let
\begin{align}
t^{n+1/2}= \frac{1}{2}(t^{n+1}-t^{n})
\quad \text{and} \quad
\widehat E_h^{\,n+\frac{1}{2}} = \frac{1}{2}(E_h^{n+1}+E_h^{n}).
\end{align}
%
In order to ensure the stability of the fully discrete scheme, we require that 
\begin{itemize}
 \item[(A3)] the time step $\tau$ is chosen such that  
 \begin{align*}
 (\mu^{-1} \curl\,v_h,\curl\,v_h) \le \frac{1}{\tau^2} (\varepsilon v_h,v_h)_h \qquad \forall v_h \in V_h,
 \end{align*}
\end{itemize}
which can be interpreted as an abstract CFL condition; see \cite{Joly03} for details. 
The following error estimates can then be proven via energy arguments.
\begin{theorem} \label{thm:2}
Let (A0)--(A3) hold and let $E$ denote a sufficiently smooth solution of \eqref{eq:maxwell1}--\eqref{eq:maxwell3}. 
Then the discrete approximations $E_h^n$, $n \ge 0$ of Problem~\ref{prob:full} satisfy
\begin{align*}
\max\limits_{0\le n< N}\Big(
\|\dt E(t^{n+\frac{1}{2}}) - \dtau E_h^{\,n+\frac{1}{2}}\|_{L^2(\Omega)} &+ \|\curl (E(t^{n+\frac{1}{2}}) - \widehat E_h^{\,n+\frac{1}{2}}))\|_{L^2(\Omega)}\Big)\\
&\le C(E) h^{r} + C'(E) \tau^2, 
\end{align*}
for all $0 \le t^n < T$ with rate $r$ and constant $C(E)$ from Theorem~\ref{thm:1} and 
\begin{align*}
C'(E) = \|\dt^{(4)} E\|_{L^1(0,T;H^1(\Th))}. \qquad
\end{align*}
\end{theorem}
\noindent
A detailed proof of this result will again be given in the appendix.

\subsection{Standard N\'ed\'elec elements}\label{sec:nedso}
To illustrate the applicability of the convergence results above, let us briefly discuss the approximation of \eqref{eq:maxwell1}--\eqref{eq:maxwell3} using second order N\'ed\'elec elements and inexact numerical integration.
We set 
\begin{align}
V_h = \{v_h \in H_0(\curl;\Omega) : v_h |_\K \in \N_1(\K)\} \subset P_2(\Th)^3. 
\end{align}
As inexact scalar product for Problem~\ref{prob:semi}, we choose 
\begin{align} \label{eq:isp1}
(\varepsilon v,w)_h = \sum\nolimits_\K (\varepsilon v,w)_{h,\K},
\end{align}
with $(\varepsilon v,w)_{h,\K}$ evaluated by appropriate numerical quadrature.
\begin{lemma} \label{lem:n1}
Assume that $(\cdot,\cdot)_{h,\K}$ is exact for polynomials of degree $p \ge 3$ and chosen such that (A1) is valid.  Then assumption (A2) holds with $q=1$ and $c_\sigma=0$. 
As a consequence, the estimates of Theorems~\ref{thm:1} and \ref{thm:2} hold with rate $r=2$, i.e., the method is second order accurate.
\end{lemma}
\begin{proof}
It suffices to verify assumption (A2) with $q=1$ and $c_\sigma=0$. 
Let $\pi_h^*=\pi_h^1$ denote the $L^2$-projection onto $P_1(\Th)^3$.
Since $\varepsilon$ is constant, $v_h|_\K \in P_2(\K)^3$, and the local quadrature rule is exact for polynomials of degree $p \ge 3$, we have 
\begin{align*}
	\sigma_h(\pi_h^* E,v_h) = \sum\nolimits_\K (\varepsilon \pi_h^1 E, v_h)_{h,\K} - (\varepsilon \pi_h^1 E, v_h)_\K = 0.
\end{align*}
This already concludes the proof of the lemma.
\end{proof}

\begin{remark}
If the quadrature formula $(\cdot,\cdot)_{h,\K}$ is exact for polynomials of degree $p \ge 4$, then the quadrature error is zero and the method of Problem~\ref{prob:semi} coincides with the standard finite element approximation of second order \cite{Monk92}, which is also included in our analysis. The previous lemma shows that some amount of inexact numerical integration is allowed without degrading the second order convergence. 
\end{remark}

\section{Second order finite elements with mass-lumping} \label{sec:n1ej1}

As observed in \cite{Cohen02,ElmkiesJoly97}, mass-lumping via numerical quadrature relies on the following key ingredients to be satisfied on every element $\K$: 
\begin{itemize}
	\item[$(i)$] three degrees of freedom are required for every quadrature point;
	\item[$(ii)$] sufficiently many quadrature have to be located at the boundary in order to allow for appropriate continuity of the associated basis functions.
\end{itemize}
We refer to Section~\ref{sec:defspace} for details. 
An appropriate quadrature rule is given by \cite{ElmkiesJoly97} 
\begin{align} \label{eq:qr} 
\int\nolimits_\K g(x) dx \approx |\K| \left(\sum_{i=1}^4 \frac{1}{40}\,g(v_{i,\K}) + 
\sum_{i=1}^4 \frac{9}{40}\,g(f_{i,\K})
\right),
\end{align}
where $v_{i,\K}$ are the vertices and $f_{i,\K}$ are the face midpoints of the tetrahedron $\K$; see Figure~\ref{fig:quadrule} for a graphical illustration.
By elementary computations, one can verify
\begin{lemma} \label{lem:qr}
The quadrature rule \eqref{eq:qr} is exact for polynomials of degree $p \le 3$.
\end{lemma}
From Lemma~\ref{lem:n1}, we infer that the N\'ed\'elec elements of second order with \eqref{eq:qr} as quadrature rule yield second order convergence. To satisfy condition $(i)$, Elmkies and Joly proposed an extension of the $\N_1$ space by four additional basis functions. In the following two sections, we analyze the approach of \cite{ElmkiesJoly97} and show that the method is second order convergent \textit{only} in particular cases. We then propose a modification that yields second order convergence in general.

\subsection{First order convergence of the $\EJ_1$ element}\label{sec:ej1fo}
We start by defining the extension of the N\'ed\'elec finite element space proposed in \cite{ElmkiesJoly97}, which we call the $ \EJ_1$ element in the sequel. Let $\lambda_{i,\K}$, $i=1,\ldots,4$ denote the barycentric coordinates of the tetrahedron $\K$, and consider the four functions
\begin{align} \label{eq:wl}
w_{\ell,\K} = \lambda_{i,\K} \lambda_{j,\K} \lambda_{k,\K} \nabla \lambda_{\ell,\K}  \in P_3(\K)^3,
\end{align}
with $\{i,j,k,\ell\}$ in circular permutation. Each of these functions can be associated to the face $f_{\ell,\K}$ opposite to the vertex $v_{\ell,\K}$; see Figure~\ref{fig:quadrule} for an illustration.
Note that $w_{\ell,\K}$ and has zero tangential trace on $\partial\K$ and, therefore, its extension by zero lies in $H(\curl,\Omega)$. 
Following \cite{Cohen02,ElmkiesJoly97}, we define 
\begin{align*}
\EJ_1(\K) = \N_1(\K) \oplus \text{span}\{w_{\ell,\K} : \ell=1,\ldots,4\} \subset P_3(\K)^3.
\end{align*}
and we introduce the corresponding global approximation space 
\begin{align}\label{eq:ej1space}
V_h = \{v \in H_0(\curl;\Omega) : v |_\K \in \EJ_1(\K)\}.
\end{align}
By elementary arguments, we can verify the following properties. 
\begin{lemma}\label{lem:fo}
Let $(\cdot,\cdot)_{h,\K}$ be defined by the quadrature rule \eqref{eq:qr}. Then assumptions (A0)--(A2) hold with $k=1$, $q=0$, and $c_\sigma=0$. As a consequence, the estimates of Theorems~\ref{thm:1} and \ref{thm:2} hold with $r=1$, i.e., the method is first order accurate. 
\end{lemma}
\begin{proof}
Assumptions (A0) and (A1) hold by construction. 
To verify (A2) with $q=0$ and $c_\sigma=0$, we choose $\pi_h^*=\pi_h^0$ and use that $\EJ_1(K) \subset P_3(\K)^3$ and Lemma~\ref{lem:qr}. This yields $\sigma_h(\pi_h^* E,v_h) = 0$. 
\end{proof}

Let us note that the additional basis functions \eqref{eq:wl} are cubic polynomials, which forced us to use $\pi_h^*=\pi_h^0$ in the estimate of the quadrature error. As we will indicate by numerical tests, the assertions of Lemma~\ref{lem:fo} are sharp, i.e., in general, the method only provides first order convergence.
In the following section, we will prove, however, that second order convergence can be obtained in special situations.

\subsection{Second order convergence for the $\EJ_1$ element}\label{sec:ej1so}

Let us start by summarizing some additional properties of the finite element space. 
\begin{lemma} \label{lem:dim}
$\text{dim}(\EJ_1(\K)) = 20+4$ and $\dim(\curl(\EJ_1(\K))) = 8+3$.
\end{lemma}
This means that the four additional basis functions $w_{\ell,\K}$, $\ell=1,\ldots,4$ are independent but one specific linear combination leads to a curl-free function. Since
\begin{align*}
\sum\nolimits_\ell w_{\ell\K} = \nabla b_\K, \qquad b_K=\lambda_{1,\K} \lambda_{2,\K}\lambda_{3,\K}\lambda_{4,\K}, 
\end{align*}
the sum of the basis functions is the problematic linear combination, since it is the gradient of the bubble function $b_\K \in P_4(\K)$. 
Based on this observation and Lemma~\ref{lem:dim}, we can split any function $v_\K \in \EJ_1(\K)$ into
\begin{align*}
v_\K = v_\K^{(1)} + v_\K^{(2)} + v_\K^{(3)}
\end{align*}
with $v_\K^{(1)} \in \N_1(\K)$, $v_\K^{(2)} \in \text{span}\{w_{\ell,\K}, \ell=1,\ldots,3\}$ and $v_\K^{(3)} \in \text{span}\{ \nabla b_{K}\}$. 
Moreover, this splitting is unique and direct, which allows to prove the following assertions.
\begin{lemma}\label{lem:stablesplit}
Let $v_\K \in \EJ_1(\K)$ be split into $v_\K = v_\K^{(1)} + v_\K^{(2)} + v_\K^{(3)}$ as above. Then
\begin{align*}
\|v_\K^{(i)}\|_{L^2(\K)} \le C \|v_\K\|_{L^2(\K)} \quad \text{and} \quad
\|\curl\,v_\K^{(i)}\|_{L^2(\K)} \le C \|\curl\,v_\K\|_{L^2(\K)}
\end{align*}
with a constant $C$ depending only on the shape of the element $\K$. 
For the second component, we further have 
\begin{align*}
\|\nabla v_\K^{(2)}\|_{L^2(\K)} \le C \|\curl\,v_\K^{(2)}\|_{L^2(\K)}. 
\end{align*}
\end{lemma}
\begin{proof}
The first two estimates follow from linear independence of the basis functions and a mapping argument. The last is based on the fact that the curls of the three basis functions spanning these components are linearly independent. As a consequence,  $\|\curl\,v_\K^{(2)}\|_{L^2(\K)}$ defines a norm on the space $\text{span}\{w_{\ell,\K} : \ell=1,\ldots,3\}$, and one can see that $\|\curl\,v_\K^{(2)}\|_{L^2(\K)} \le c \|\nabla v_\K^{(2)}\|_{L^2(\K)}$, hence also $\|\nabla v_\K^{(2)}\|_{L^2(\K)}$ defines a norm on this three dimensional subspace. The fact that $c$ and $C$ can be chosen depending only on the shape of the element follows from the usual mapping argument and the uniform shape regularity of the mesh.
\end{proof}

\begin{remark}\label{rem:split}
The above splitting generalizes directly to the whole space, i.e., any function $v_h \in V_h$ defined in \eqref{eq:ej1space} can be split uniquely into
\begin{align*}
v_h = v_h^{(1)} + v_h^{(2)} + v_h^{(3)}
\end{align*}
with components given by
\begin{align*}
v_h^{(1)} &\in V_h^{(1)} = \{v \in H_0(\curl,\Omega) : v|_\K \in \N_1(\K)\},\\
v_h^{(2)} &\in V_h^{(2)} = \{v \in H_0(\curl,\Omega) : v|_\K \in \text{span}\{w_{\ell,\K} : \ell=1,\ldots,3\}\},\\
v_h^{(3)} &\in V_h^{(3)} = \{v \in H_0(\curl,\Omega) : v|_\K \in \text{span}\{b_\K\}\}.
\end{align*}
Since we assumed uniform shape regularity of the mesh, we can take the same constant $C$ in Lemma~\ref{lem:stablesplit} and thus obtain global estimates
\begin{align*}
\|v_h^{(i)}\|_{L^2(\Omega)} \le C \|v_h\|_{L^2(\Omega)},
\qquad
\|\curl\,v_h^{(i)}\|_{L^2(\Omega)} &\le C \|\curl\,v_h\|_{L^2(\Omega)}, \quad i=1,\ldots,3\\
\text{and} \qquad 
\|\nabla v_h^{(2)}\|_{L^2(\Omega)} &\le C \|\curl\,v_h^{(2)}\|_{L^2(\Omega)}, 
\end{align*}
which follow by summation of the local estimates of the previous lemma. 
\end{remark}
In the sequel, we utilize an additional divergence preserving projection operator. Let $\pi_\K^{\BDM} : H^1(\K)^3 \to \BDM_1(\K)$ be the canonical projection operator for the space $\BDM_1(\K)=P_1(\K)^3$, see \cite[Sec 2.5]{BoffiBrezziFortin13} or  \cite{BrezziDouglasMarini85}, and recall that 
\begin{align*}
\div (\pi_\K^{\BDM} v) = \pi_\K^0 \div\,v \qquad \text{for all } v \in H^1(\K)^3. 
\end{align*}
Moreover, the following approximation error estimates hold
\begin{alignat*}{2}
\|\pi_\K^{\BDM} v - v\|_{L^2(\K)}  &\le C h^s \|v\|_{H^s(\K)}, \qquad &1 \le s \le 2 \\
\|\div (\pi_\K^{\BDM} v - v)\|_{L^2(\K)} &\le C h \|\div\,v\|_{H^1(\K)}.
\end{alignat*}
By $(\pi_h^{\BDM} v)|_K = \pi_K^{\BDM} v|_K$ 
we define the corresponding global projection operator for piecewise smooth functions $v \in H^1(\Th)^3$.
With the help of this projection operator and the splitting of the test space $V_h$, we can now establish the following 
improved estimates for the quadrature error. 
\begin{lemma}\label{lem:so}
Let $\div(\varepsilon E) = 0$. 
Then assumptions (A0)-(A2) hold with $k=1$, $q=1$, and some $c_\sigma>0$ depending only on the shape regularity of the mesh. As a consequence, the estimates of Theorems~\ref{thm:1} and \ref{thm:2} hold with rate $r=2$, i.e., the method is second order accurate if the exact solution is divergence free.
\end{lemma}
\begin{proof}
We choose $\pi_h^*=\pi_h^\BDM$ and let $v_h = v_h^{(1)} + v_h^{(2)} + v_h^{(3)}$ be the splitting of $v_h \in V_h$ defined above. Then
\begin{align*}
\sigma_h(\pi_h^* E,v_h) 
&= \sigma_h(\pi_h^* E, v_h^{(1)}) + \sigma_h(\pi_h^* E, v_h^{(2)}) + \sigma_h(\pi_h^* E, v_h^{(3)})
 = (i) + (ii) + (iii).
\end{align*}
Due to the exactness of the quadrature rule stated in Lemma~\ref{lem:qr}, we have $(i)=0$. For the second term, we get
\begin{align*}
(ii) 
&= \sigma_h(\pi_h^* E,v_h^{(2)}) \\
&= \sigma_h(\pi_h^* E - \pi_h^0 E, v_h^{(2)} - \pi_h^0 v_h^{(2)}) + \sigma_h(\pi_h^0 E, v_h^{(2)} - \pi_h^0 v_h^{(2)}) + \sigma_h(\pi_h^* E, \pi_h^0 v_h^{(2)}).
\end{align*}
By Lemma~\ref{lem:qr}, 
the last two terms vanish identically, and we obtain
\begin{align*}
(ii) 
\le C h^2 \|E\|_{H^1(\Th)} \|\nabla v_h^{(2)}\|_{L^2(\Th)} 
\le C' h^2 \|E\|_{H^1(\Th)} \|\curl\,v_h^{(2)}\|_{L^2(\Th)}, 
\end{align*}
where we used the approximation properties of $\pi_h^\BDM$ and the third estimate of Remark~\ref{rem:split} in the last step. This is the required estimate for the second component. 
From the assumption $\div(\varepsilon E)=0$ and the divergence-preserving property of the projection operator $\pi_h^*=\pi_h^{\BDM}$, we infer that 
\begin{align*}
\pi_\K^* E \in H^1(\K) 
\qquad \text{and} \qquad 
\div(\pi_\K^* E) = 0 \qquad \forall \K \in \Th. 
\end{align*}
A close inspection of the quadrature rule  \eqref{eq:qr} shows that 
\begin{align*}
\sigma_\K(w_\K,v_\K^{(3)}) = 0 \qquad \text{for all } w_\K \in P_1(\K) \quad \text{ with }\quad \div\,w_\K=0
\end{align*}
and $v_\K^{(3)} \in \text{span}\{\nabla b_K\}$ as defined above, i.e. the quadrature rule also integrates one additional forth order polynomial exactly.
As a consequence, $(iii)=0$, and the proof is concluded by summing up the estimates for the terms $(i)$--$(ii)$.
\end{proof}

\begin{remark}
The $\EJ_1$ element with mass-lumping thus yields second order convergence provided that the exact solution is divergence free. From \eqref{eq:maxwell1}, we see that 
\begin{align*}
\div (\varepsilon \dtt E(t)) 
&= -\div(\curl(\mu^{-1} \curl\,E(t))) - \div(\dt j(t)) =- \div(\dt j(t)).   
\end{align*}
Hence the assumption $\div(\varepsilon E(t))=0$ holds, if $\div(\varepsilon E(0))=\div(\varepsilon \dt E(0))=0$ and $\div\,j(t)=0$ for all $t \ge 0$. 
In this case, Lemma~\ref{lem:so} guarantees second order convergence, which explains the good numerical results obtained in \cite{Cohen02,ElmkiesJoly97}. 
%
\end{remark}

\subsection{A modification of the $\EJ_1$ element}\label{sec:42}

From the error analysis presented in the previous section, one can see that the problematic component in the estimate for the quadrature error is that related with the test function $v_h^{(3)} \in V_h^{(3)} = \{v \in H_0(\curl;\Omega) : v|_K \in \text{span}\{\nabla b_K\}\}$ which cannot be controlled via its curl nor is integrated with sufficient accuracy in the general case. 
We therefore replace the basis functions $w_{4,\K}$ in the $\EJ_1$ element by 
\begin{align} \label{eq:w4s}
w_{4,\K}^* = w_{4,\K} +\lambda_{1,\K} \lambda_{2,\K} \lambda_{3,\K} (\lambda_{2,\K}-\lambda_{1,\K})\nabla \lambda_{4,\K}, 
\end{align}
and define the modified $\EJ_1$ element by
\begin{align*}
\mEJ_1(\K) = \N_1(\K) \oplus \text{span}\{w_{1,\K},w_{2,\K},w_{3,\K},w_{4,\K}^*\}. 
\end{align*}
By elementary computations, one can again verify the following properties. 
\begin{lemma}
$\text{dim}(\mEJ_1(\K)) = 20+4$ and $\text{dim}(\curl(\mEJ_1(\K)))=8+4$, i.e.,
the four additional functions have linear independent curls. 
Moreover, $w_{4,\K}^*$ is integrated exactly by the quadrature rule \eqref{eq:qr}, i.e.,  $\sigma_\K(p_K^0,w_{4,\K}^*)=0$ for all $p_K^0 \in P_0(\K)^3$. 
\end{lemma}
As approximation space for Problems~\ref{prob:semi} and \ref{prob:full}, 
we now consider
\begin{align*}
V_h = \{v \in H_0(\curl;\Omega) : v|_\K \in \mEJ_1(\K)\}. 
\end{align*}
Based on the previous results, we can prove the following assertion. 
\begin{lemma}\label{lem:so2}
Assumptions (A0)--(A2) hold with $k=1$, $q=1$, and $c_\sigma>0$ only depending on the shape-regularity of the mesh. Hence the assertions of Theorems~\ref{thm:1} and \ref{thm:2} hold with rate $r=2$, i.e., the method is second order accurate.
\end{lemma}
\begin{proof}
Conditions (A0) and (A1) follow by construction. 
For the proof of property (A2), we note that $v_h$ can now be split into 
$v_h = v_h^{(1)} + v_h^{(2)}$
with components
\begin{align*}
v_h^{(1)} &\in V_h^{(1)} = \{v \in H_0(\curl,\Omega) : v|_\K \in \N_1(\K)\}\quad\text{and}\\
v_h^{(2)} &\in V_h^{(2)} = \{v \in H_0(\curl,\Omega) : v|_\K \in \text{span}\{w_{1},w_{2},w_{3},w_{4}^*\}.
\end{align*}
We can then continue as in the proof of Lemma~\ref{lem:so} without taking the third component in the decomposition of the quadrature error into consideration.
\end{proof}

In the next section, we define a local basis for the space $\mEJ_1$, which together with the quadrature rule \eqref{eq:qr} leads to a block-diagonal mass matrix $\widetilde M$. 
We thus obtain an efficient method of second order accuracy.

\section{Definition of the basis functions} \label{sec:defspace}
We start by defining basis functions for the standard Nédélec space $\N_1(K)$.
We only consider a single element $\K$ and omit the corresponding subscript in the following. 
Let $\lambda_i$, $i=1,\ldots,4$ denote the barycentric coordinates of the element $K$ associated to the vertices $v_i$, which are ordered with respect to their global index in the mesh. Furthermore, let $f_i$, $i=1,\ldots,4$ denote the midpoint of the face opposite to the vertex $v_i$, see Figure~\ref{fig:quadrule}.
For each edge spanned by vertices $v_i$ and $v_j$ with $i \ne j$ we define two basis functions of the form $\lambda_i\nabla\lambda_j$ leading to
\begin{alignat*}{4}
& \Phi_1 =  \lambda_1\nabla\lambda_2, \qquad 
&&\Phi_2 = \lambda_2\nabla\lambda_1, \qquad
&&\Phi_3 =  \lambda_1\nabla\lambda_3, \qquad 
&&\Phi_4 = \lambda_3\nabla\lambda_1, \\
& \Phi_5 =  \lambda_1\nabla\lambda_4, \qquad 
&&\Phi_6 = \lambda_4\nabla\lambda_1, \qquad
&&\Phi_7 =  \lambda_2\nabla\lambda_3, \qquad 
&&\Phi_8 = \lambda_3\nabla\lambda_2, \\
&    \Phi_9 =  \lambda_2\nabla\lambda_4, \qquad 
&&\Phi_{10} = \lambda_4\nabla\lambda_2, \qquad
&&\Phi_{11} =  \lambda_3\nabla\lambda_4, \qquad 
&&\Phi_{12} = \lambda_4\nabla\lambda_3.
\end{alignat*}
Any basis function associated to an edge has non-zero tangential trace only on the respective edge and vanishes in all but one vertex. 
For a face with vertices $v_i$, $v_j$, $v_k$, $i < j,k$, $j \ne k$, we define two basis functions of the form $\lambda_j(\lambda_i\nabla\lambda_k-\lambda_k\nabla\lambda_i)$, viz.
\begin{alignat*}{4}
&\Phi_{13} = \lambda_3(\lambda_2\nabla\lambda_4-\lambda_4\nabla\lambda_2),\qquad\qquad
&&\Phi_{14} = \lambda_4(\lambda_2\nabla\lambda_3-\lambda_3\nabla\lambda_2), \\
&\Phi_{15} = \lambda_2(\lambda_1\nabla\lambda_4-\lambda_4\nabla\lambda_1),\qquad
&&\Phi_{16} = \lambda_4(\lambda_1\nabla\lambda_2-\lambda_2\nabla\lambda_1), \\
&\Phi_{17} = \lambda_2(\lambda_1\nabla\lambda_4-\lambda_4\nabla\lambda_1),\qquad
&&\Phi_{18} = \lambda_4(\lambda_1\nabla\lambda_2-\lambda_2\nabla\lambda_1), \\
&\Phi_{19} = \lambda_2(\lambda_1\nabla\lambda_3-\lambda_3\nabla\lambda_1),\qquad
&&\Phi_{20} = \lambda_3(\lambda_1\nabla\lambda_2-\lambda_2\nabla\lambda_1).
\end{alignat*}
These functions vanish identically at all vertices $v_\ell$ and on one of the faces. 
The functions $\{\Phi_i : 1 \le i \le 20\}$ form a basis for the N\'ed\'elec space $\N_1(K)$. 
Note that three basis functions can be associated to each vertex while \textit{only} two functions are associated to every face midpoint; see Figure~\ref{fig:quadrule} for an illustration. 
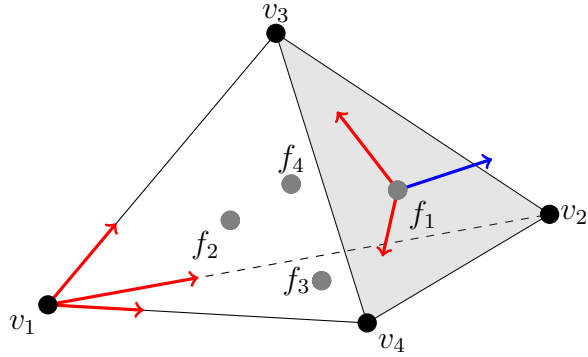
\begin{figure}[ht!]
\begin{tikzpicture}[scale=1.2]

\def \factor {0.3}
\def \offset {0cm}


\coordinate (A) at (-0.5,1); 
\coordinate (B) at (5,2);
\coordinate (C) at (2,4);
\coordinate (D) at (3,0.8);

\node at (A) [anchor = north east] {$v_1$};
\node at (B) [anchor = west] {$v_2$};
\node at (C) [anchor = south] {$v_3$};
\node at (D) [anchor = north west] {$v_4$};

\draw[dashed] (A) -- (B);
\draw (A) -- (C);
\draw (B) -- (C);
\draw (A) -- (D);
\draw (D) -- (B);
\draw (C) -- (D);
\fill [gray,opacity=0.2] (B) -- (C) -- (D);

\coordinate (V1) at (A);
\coordinate (V2) at (C);
\coordinate (N) at ($(V1)-\factor*(V1)+\factor*(V2)$);
\coordinate (E) at ($(V1)!\offset!90:(N)$);
\coordinate (F) at ($(N)!\offset!-90:(V1)$);
\draw[red, very thick,->] (E) -- (F);

\coordinate (V1) at (A);
\coordinate (V2) at (D);
\coordinate (N) at ($(V1)-\factor*(V1)+\factor*(V2)$);
\coordinate (E) at ($(V1)!\offset!-90:(N)$);
\coordinate (F) at ($(N)!\offset!90:(V1)$);
\draw[red, very thick,->] (E) -- (F);

\coordinate (V1) at (A);
\coordinate (V2) at (B);
\coordinate (N) at ($(V1)-\factor*(V1)+\factor*(V2)$);
\coordinate (E) at ($(V1)!\offset!90:(N)$);
\coordinate (F) at ($(N)!\offset!-90:(V1)$);
\draw[red, very thick,->] (E) -- (F);

\draw[fill,black] (C) circle (0.1cm);
\draw[fill,black] (B) circle (0.1cm);
\draw[fill,black] (D) circle (0.1cm);
\draw[fill,black] (A) circle (0.1cm);

\coordinate (ABC) at ($1/3*(A)+1/3*(B)+1/3*(C)$);
\node at (ABC) [anchor = south] {$f_4$};

\coordinate (ABD) at ($1/3*(A)+1/3*(B)+1/3*(D)$);
\node at (ABD) [anchor = east] {$f_3$};

\coordinate (ACD) at ($1/3*(A)+1/3*(C)+1/3*(D)$);
\node at (ACD) [anchor = north east] {$f_2$};

\coordinate (BCD) at ($1/3*(B)+1/3*(C)+1/3*(D)$);
\node at (BCD) [anchor = north west] {$f_1$};

\def \factor {0.5}

\coordinate (V1) at (BCD);
\coordinate (V2) at (D);
\coordinate (N) at ($(V1)-\factor*(V1)+\factor*(V2)$);
\draw[red, very thick,->] (V1) -- (N);

\draw[fill,gray] (ABC) circle (0.1cm);
\draw[fill,gray] (ABD) circle (0.1cm);
\draw[fill,gray] (ACD) circle (0.1cm);
\draw[fill,gray] (BCD) circle (0.1cm);

\coordinate (V1) at (BCD);
\coordinate (V2) at (C);
\coordinate (N) at ($(V1)-\factor*(V1)+\factor*(V2)$);
\draw[red, very thick,->] (V1) -- (N);

\draw[fill,gray] (ABC) circle (0.1cm);
\draw[fill,gray] (ABD) circle (0.1cm);
\draw[fill,gray] (ACD) circle (0.1cm);
\draw[fill,gray] (BCD) circle (0.1cm);

\def \factor {-0.27}
\coordinate (V1) at (BCD);
\coordinate (V2) at (A);
\coordinate (N) at ($(V1)-\factor*(V1)+\factor*(V2)$);
\draw[blue, very thick,->] (V1) -- (N);

\draw[fill,gray] (ABC) circle (0.1cm);
\draw[fill,gray] (ABD) circle (0.1cm);
\draw[fill,gray] (ACD) circle (0.1cm);
\draw[fill,gray] (BCD) circle (0.1cm);

\end{tikzpicture}
\caption{Quadrature points and degrees of freedom of the standard Nédélec element $\N_1$ (red) and the $\EJ_1$ element (red and blue).\label{fig:quadrule}}
\end{figure}%

As observed in \cite{ElmkiesJoly97}, we require three basis functions for each quadrature point in order to allow for mass-lumping. We therefore define an additional third basis function for every face, see Section~\ref{sec:42} and Figure~\ref{fig:quadrule}.
\begin{align*}
\widehat\Phi_{21} &= \lambda_2\lambda_3\lambda_4\nabla \lambda_1, \qquad \qquad 
\widehat\Phi_{22} = \lambda_1\lambda_3\lambda_4\nabla \lambda_2, \\
\widehat\Phi_{23} &= \lambda_1\lambda_2\lambda_4\nabla \lambda_3, \qquad \qquad 
\widehat\Phi_{24} = \lambda_1\lambda_2\lambda_3(1+\lambda_2-\lambda_1)\nabla \lambda_4.
\end{align*}
Any of these functions vanishes identically on three of the faces and has zero tangential component on the remaining face. Therefore, any function $\widehat\Phi_i$, $21\le i\le 24$ is non-zero in only one of the quadrature points of the quadrature rule \eqref{eq:qr}.
In order to achieve mass-lumping, we now modify the other basis functions $\Phi_i$, $1 \le i \le 20$ in order to obtain a similar property. We first modify the face basis functions by
\begin{align*}
	\widehat \Phi_{j} = \Phi_{j} + \sum_{i=1}^4 a_{i,j}\, \widehat \Phi_{20+i}, \qquad 13 \le j \le 20,
\end{align*}
with coefficient matrix defined by
{\footnotesize
\begin{align*}
	a = \begin{pmatrix}
	0 & 0 & 3 & 3 & 3 & 3 & 3 & 3 \\
	3 & 3 & 0 & 0 & 0 &-3 & 0 &-3 \\
	0 &-3 & 0 &-3 & 0 & 0 &-3 & 0 \\
   -3 & 0 &-3 & 0 &-3 & 0 & 0 & 0
	\end{pmatrix}.
\end{align*}}%
Any of these functions is non-zero in only one of the quadrature points of \eqref{eq:qr}. In a second step, we modify the edge basis functions by
\begin{align*}
\widehat \Phi_{j} = \Phi_{j} + \sum_{i=1}^{12} b_{i,j}\, \widehat \Phi_{12+i}, \qquad 1 \le j \le 12,
\end{align*}
with the following coefficient matrix
{\footnotesize
\begin{align*}
b = \begin{pmatrix}
	0 & 0 & 0 & 0 & 0 & 0 & 1 & 1 &-2 & 1 &-2 & 1\\
	0 & 0 & 0 & 0 & 0 & 0 &-2 & 1 & 1 & 1 & 1 &-2\\
	0 & 0 & 1 & 1 &-2 & 1 & 0 & 0 & 0 & 0 &-2 & 1\\
	0 & 0 &-2 & 1 & 1 & 1 & 0 & 0 & 0 & 0 & 1 &-2\\
	1 & 1 & 0 & 0 &-2 & 1 & 0 & 0 &-2 & 1 & 0 & 0\\
   -2 & 1 & 0 & 0 & 1 & 1 & 0 & 0 & 1 &-2 & 0 & 0\\
    1 & 1 &-2 & 1 & 0 & 0 &-2 & 1 & 0 & 0 & 0 & 0\\
   -2 & 1 & 1 & 1 & 0 & 0 & 1 &-2 & 0 & 0 & 0 & 0\\
    0 &-9 & 0 &-9 & 0 &-9 & 3 & 3 & 3 & 3 & 3 & 3\\
   -9 & 0 & 3 & 3 & 3 & 3 & 0 &-9 & 0 &-9 & 3 & 3\\
    3 & 3 &-9 & 0 & 3 & 3 &-9 & 0 & 3 & 3 & 0 &-9\\
    3 & 3 & 3 & 3 &-9 & 0 & 3 & 3 &-9 & 0 &-9 & 0
\end{pmatrix}.
\end{align*}}%
By construction, any of the basis functions $\widehat \Phi_j$, $1 \le j \le 12$ now vanishes in all 
integration points of the quadrature rule \eqref{eq:qr} except one. 
\begin{remark}
The basis functions $\widehat \Phi_j$, $1 \le j \le 24$ form a basis of $\mEJ_1(K)$ and at any quadrature point of \eqref{eq:qr}, exactly three basis functions are non-zero. The local mass matrix $\widehat M^K$ with entries $(\widehat M^K)_{ij} = (\widehat \Phi_i, \widehat \Phi_j)_{h,K}$, obtained by inexact numerical integration, is regular and block-diagonal with $3 \times 3$ blocks, one for each quadrature point.
The global mass matrix $\widehat M$ obtained by assembling of the local matrices is also block-diagonal, with one block for each vertex and each face midpoint of the mesh. The size of the vertex blocks is determined by the number of edges adjacent to a vertex, while any of the face blocks has size $4$; see Figure~\ref{fig:bldiag} for an example.
The proposed inexact numerical integration, together with the above choice of basis functions can therefore be interpreted as a mass-lumping strategy.
\end{remark}
\begin{figure}
	\centering
	\includegraphics[scale=0.5]{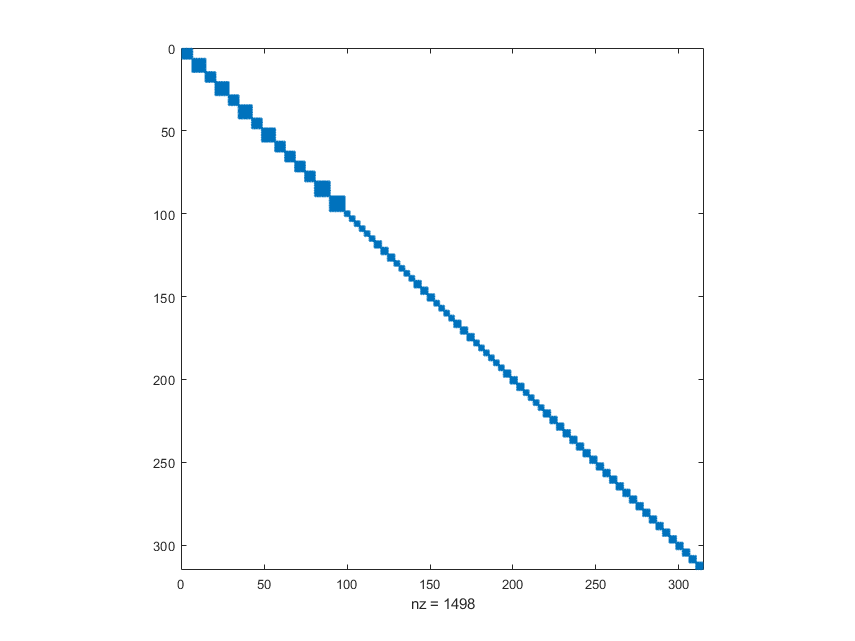}
	\caption{Structure of the $\mEJ_1$ lumped mass matrix on a mesh $\Th$ with $24$ elements.\label{fig:bldiag}}
\end{figure}
%
%


%
%
%
%

\section{Numerical tests} \label{sec:num}

In this section, we conduct two numerical experiments. First, we design a test case where the solution $E(t)$ is divergence-free for all $t$. In this case, both $\EJ_1$ and our modified element $\mEJ_1$ are expected to exhibit second order convergence, as stated in Lemma~\ref{lem:so} and Lemma~\ref{lem:so2}. As a second test case, we choose a solution which \emph{is not} divergence free. In this case, we will see that the $\EJ_1$ element only yields first order convergence, i.e. the result of Lemma~\ref{lem:fo} is sharp, while our modified element $\mEJ_1$ still retains second order convergence. We also briefly investigate the CFL-conditions in assumption (A3).

In all numerical tests, we set $\Omega=(0,1)^3$ and let $0\le t\le T=2$ denote the time interval, and we set $\varepsilon=\mu=1$. We specify different exact solutions $E$ for the different test cases and define the corresponding right hand sides by
\begin{align}
	(f(t^n),v_h) = (\dtt E(t^n),v_h)+(\curl\,E(t^n),\curl\,v_h).
\end{align}
The fully discrete solution is computed by Problem~\ref{prob:full} with initial values $E_h^0=\widehat \Pi_h E(t^0)$ and $E_h^1=\widehat \Pi_h E(t^1)$ defined by the elliptic projection
\begin{align}\label{eq:elpr}
(\widehat \Pi_h E(t^n),v_h)+(\curl\,\widehat \Pi_h E(t^n),\,\curl\,v_h)=(E(t^n),v_h)+(\curl\,E(t^n),\curl\,v_h),
\end{align}
for $n=0,1$. This guarantees that the initial approximations are second order accurate in the $H(\curl)$ norm, so they do not impact the overall convergence.
For this choice of initial values, we will actually observe some super-convergence phenomena.

\subsection{Divergence-free solution}
For the first test case, we choose
\begin{align*}
E(x,y,z,t)&=  \cos(t)\begin{pmatrix} -\sin(\pi x)\cos(\pi y) \\ \cos(\pi x)\sin(\pi y) \\ 0 \end{pmatrix}
\end{align*}
as the exact solution, which satisfies $\div\,E=0$. 
We denote by $E_h$ and $E_h^*$ the discrete solution generated by the discretization via the element $\EJ_1$ and $\mEJ_1$, respectively. 
For the convergence study, we will consider a sequence $\{\Th\}_{h}$ of quasi-uniform but non-nested meshes $\Th$ with decreasing mesh size $h \approx 2^{-k} =: h(k)$, $k \ge 1$. For all computations, we choose $\tau = 0.02 h(k)$ as the time step size, and we use 
$$
\tnorm e\tnorm=\max_{0 \le t^n < T} \|e(t^n)\|_{L^2(\Omega)}
$$ 
to measure the error.
In Table~\ref{tab:1} and Table~\ref{tab:2}, we display the discrete errors obtained by Problem~\ref{prob:full} with respect to the elliptic projection \eqref{eq:elpr} and the estimated orders of convergence (eoc).
\begin{table}[ht!]
	\small
	\begin{tabular}{c|c||c|c||c|c} 
		$h(k)$ & $\#$dof & $\tnorm \widehat \Pi_h E - E_h\tnorm$ & eoc & $\tnorm \curl\,(\widehat\Pi_h E - E_h)\tnorm$ & eoc \\
		\hline
		\hline
		$2^{-1}$ & $1400$ & $0.018107$ & ---    & $0.094814$ & ---      \\
		$2^{-2}$ & $5698$ & $0.006319$ & $1.51$ & $0.036393$ & $1.38$   \\
		$2^{-3}$ & $44094$ & $0.001635$ & $1.95$ & $0.008899$ & $2.03$  \\
		$2^{-4}$ & $344408$ & $0.000417$ & $1.97$ & $0.002244$ & $1.99$ \\
		$2^{-5}$ & $2802226$ & $0.000105$ & $1.99$ & $0.000547$ & $2.04$
	\end{tabular}
	\medskip
	\caption{Discrete errors for the $\EJ_1$ element resulting from Problem~\ref{prob:full}.\label{tab:1}} 
	
	\begin{tabular}{c|c||c|c||c|c} 
		$h(k)$ & $\#$dof & $\tnorm \widehat \Pi_h E - E_h^*\tnorm$ & eoc & $\tnorm \curl\,(\widehat \Pi_h E - E_h^*)\tnorm$ & eoc  \\
		\hline
		\hline
		$2^{-1}$ & $1400$ & $0.008811$ & ---    & $0.095367$ & ---      \\
		$2^{-2}$ & $5698$ & $0.001204$ & $2.87$ & $0.036455$ & $1.39$   \\
		$2^{-3}$ & $44094$ & $0.000104$ & $3.52$ & $0.008924$ & $2.03$  \\
		$2^{-4}$ & $344408$ & $0.000012$ & $3.01$ & $0.002249$ & $1.99$ \\
		$2^{-5}$ & $2802226$ & $0.000001$ & $3.02$ & $0.000548$ & $2.04$
	\end{tabular}
	\medskip
	\caption{Discrete errors for the $\mEJ_1$ element resulting from Problem~\ref{prob:full}.\label{tab:2}} 
\end{table}
As predicted by theory, both methods have second order convergence in the energy norm. Let us note that the discrete error of the modified $\mEJ_1$ element exhibits super-convergence in the $L^2$-norm.

\subsection{Non divergence-free solution}
For the second convergence test, let
\begin{align*}
E(x,y,z,t)&=  \cos(t)\begin{pmatrix} -\sin(\pi x)\cos(\pi y) \\ \cos(\pi x)\cos(\pi y) \\ 0 \end{pmatrix}
\end{align*}
be the exact solution. Note that $\div\,E\neq 0$ in this case. 
In Tables~\ref{tab:3} and \ref{tab:4}, we depict the discrete errors and estimated orders of convergence obtained for Problem~\ref{prob:full}.
\begin{table}[ht!]
	\small
	\begin{tabular}{c|c||c|c||c|c} 
		$h(k)$ & $\#$dof & $\tnorm \widehat\Pi_h E - E_h\tnorm$ & eoc & $\tnorm \curl\,(\widehat\Pi_h E - E_h)\tnorm$ & eoc \\
		\hline
		\hline
		$2^{-1}$ & $1400$    & $0.073937$ & ---    & $0.073391$ & ---    \\
		$2^{-2}$ & $5698$    & $0.041814$ & $0.82$ & $0.025524$ & $1.52$ \\
		$2^{-3}$ & $44094$   & $0.020204$ & $1.05$ & $0.006218$ & $2.04$ \\
		$2^{-4}$ & $344408$  & $0.009990$ & $1.02$ & $0.001578$ & $1.98$ \\
		$2^{-5}$ & $2802226$ & $0.004945$ & $1.01$ & $0.000381$ & $2.05$
	\end{tabular}
	\medskip
	\caption{Discrete errors for the $\EJ_1$ element resulting from Problem~\ref{prob:full}.\label{tab:3}} 
	\begin{tabular}{c|c||c|c||c|c} 
		$h(k)$ & $\#$dof & $\tnorm \widehat\Pi_h E - E_h^*\tnorm$ & eoc & $\tnorm \curl\,(\widehat\Pi_h E - E_h^*)\tnorm$ & eoc  \\
		\hline
		\hline
		$2^{-1}$ & $1400$    & $0.009595$ & ---    & $0.073786$ & ---    \\
		$2^{-2}$ & $5698$    & $0.002358$ & $2.02$ & $0.025548$ & $1.53$ \\
		$2^{-3}$ & $44094$   & $0.000290$ & $3.02$ & $0.006520$ & $1.97$ \\
		$2^{-4}$ & $344408$  & $0.000035$ & $3.02$ & $0.001701$ & $1.94$ \\
		$2^{-5}$ & $2802226$ & $0.000004$ & $3.05$ & $0.000416$ & $2.03$
	\end{tabular}
	\medskip
	\caption{Discrete errors for the $\mEJ_1$ element resulting from Problem~\ref{prob:full}.\label{tab:4}} 
\end{table}
From the results of Table~\ref{tab:3}, we infer that the estimate of Lemma~\ref{lem:fo} is indeed sharp, i.e., the $\EJ_1$ element only exhibits first order convergence. Note, however, that the error for the $\curl$ is still second order convergent. 
As predicted, the modified element $\mEJ_1$ yields second order convergence and again shows super-convergence in the $L^2$-norm, just as in our first test case.

\subsection{CFL-condition}

To further evaluate the fully discrete method of Problem~\ref{prob:full}, we now 
investigate the CFL-condition resulting from assumption (A3). We are thus looking for a constant $c>0$ such that
\begin{align}\label{eq:cflc}
\tau\le \tau_{\max}=c\cdot h(k).
\end{align} 
Let $\widetilde M$ and $K$ denote the mass and stiffness matrices resulting from space discretization in Problems~\ref{prob:semi} and \ref{prob:full}; see \eqref{eq:varh_ls}. 
From the proof of Theorem~\ref{thm:2}, we see that a sufficient condition for discrete stability is
\begin{align*}
\frac{\tau_{\max}^2}{4}\|M_h^{-1}K\|=\frac{1}{4}.
\end{align*}
By plugging this in \eqref{eq:cflc}, we obtain
\begin{align}
c = \frac{1}{h(k)\sqrt{\lambda_{\max}(M_h^{-1}K)}}\label{eq:cfl}.
\end{align}
In Table~\ref{tab:5}, we compare the resulting constants for the space discretizations based on the standard $\N_1$ element without and the $\EJ_1$ and the modified $\mEJ_1$ element with mass lumping. 
\begin{table}[ht!]
	\setlength{\tabcolsep}{15pt}
	\begin{tabular}{c||c|c|c} 
		$h$ & $\N_1$ & $\EJ_1$  & $\mEJ_1$ \\
		\hline
		\hline
		$2^{-1}$ & $0.045386$ & $0.040525$ & $0.040454$  \\
		$2^{-2}$ & $0.051658$ & $0.046686$ & $0.046559$  \\
		$2^{-3}$ & $0.048364$ & $0.042656$ & $0.042551$  \\
		$2^{-4}$ & $0.048310$ & $0.042687$ & $0.042583$  \\
		$2^{-5}$ & $0.048443$ & $0.043052$ & $0.042963$
	\end{tabular}
	\medskip
	\caption{Values of the CFL-constant $c$ in \eqref{eq:cfl} for the standard $\N_1$ discretization \textit{without} mass-lumping, and the $\EJ_1$ and $\mEJ_1$ elements \textit{with} mass-lumping; larger is better. \label{tab:5}}
\end{table}%

Let us note that all discretizations yield very similar CFL-constants.

\section{Discussion}\label{sec:discussion}

In this paper, we considered inexact Galerkin finite element approximations for Maxwell's equations. We established second order convergence for the discretizations proposed by Elmkies and Joly in \cite{ElmkiesJoly97} for divergence free solutions, and illustrated that, in general, the method is only first order accurate. 
A slight modification of the finite element space allowed us to 
obtain a method which is second order accurate in the general case. 
For the modified element $\mEJ_1$, super-convergence was observed for the difference of the numerical solution and the elliptic projection of the true solution. In addition, we also observed super-convergence for the error in the $\curl$ for the original $\EJ_1$ element. A theoretical explanation for these observations is still open. 
Let us mention that similar arguments as used in the analysis of Sections~\ref{sec:prelim} and \ref{sec:n1ej1} can also be applied to wave propagation problems in $H(\div)$ and $H^1$; see \cite{EggerRadu18,EggerRadu18b,GeeversMulderVegt18}.
In principle, also the extension of our arguments to higher order is possible. Finding appropriate quadrature rules of higher order, however, is not trivial. Moreover, the additional number of degrees of freedom needed to enable mass-lumping increases strongly with the approximation order. In that case, discontinuous Galerkin methods seem advantageous. 
Some comparison of mass-lumped finite elements and discontinuous Galerkin methods has been conducted in \cite{GeeversMulderVegt18} for problems in $H^1$.

\small

\section*{Acknowledgements}

The authors are grateful for financial support by the ``Excellence Initiative'' of the German Federal and State Governments via the Graduate School of Computational Engineering GSC~233 at Technische Universität Darmstadt and by the German Research Foundation (DFG) via grants IRTG~1529, TRR~146 project C3, and TRR~154 project C4. 


\begin{thebibliography}{10}
	
	\bibitem{BoffiBrezziFortin13}
	D.~Boffi, F.~Brezzi, and M.~Fortin.
	\newblock {\em Mixed finite element methods and applications}, volume~44 of
	{\em Springer Series in Computational Mathematics}.
	\newblock Springer, Heidelberg, 2013.
	
	\bibitem{BrezziDouglasMarini85}
	F.~Brezzi, J.~Douglas, and L.~D. Marini.
	\newblock Two families of mixed elements for second order elliptic problems.
	\newblock {\em Numer. Math.}, 88:217--235, 1985.
	
	\bibitem{CodecasaPoliti08}
	L.~Codecasa and M.~Politi.
	\newblock Explicit, consistent, and conditionally stable extension of {FD}-{TD}
	to tetrahedral grids by {FIT}.
	\newblock {\em IEEE Trans. Magn.}, 44:1258--1261, 2008.
	
	\bibitem{Cohen02}
	G.~Cohen.
	\newblock {\em Higher-Order Numerical Methods for Transient Wave Equations}.
	\newblock Springer, Heidelberg, 2002.
	
	\bibitem{CohenMonk98}
	G.~Cohen and P.~Monk.
	\newblock Gauss point mass lumping schemes for {M}axwell's equations.
	\newblock {\em Numer. Meth. Part. Diff. Equat.}, 14:63--88, 1998.
	
	\bibitem{Dupont73}
	T.~Dupont.
	\newblock ${L}^2$ estimates for {G}alerkin methods for second-order hyperbolic
	equations.
	\newblock {\em SIAM J. Numer. Anal.}, 10:880--889, 1973.
	
	\bibitem{EggerRadu18}
	H.~Egger and B.~Radu.
	\newblock Super-convergence and post-processing for mixed finite element
	approximations of the wave equation.
	\newblock Technical report.
	
	\bibitem{EggerRadu18b}
	H.~Egger and B.~Radu.
	\newblock A mass-lumped mixed finite element method for acoustic wave
	propagation.
	\newblock 2018.
	\newblock arXive:1803.04238.
	
	\bibitem{EggerRadu19}
	H.~Egger and B.~Radu.
	\newblock A mass-lumped mixed finite element method for {M}axwell's equations.
	\newblock {\em arXiv:1810.06243}, 2018.
	\newblock to appear in Proceedings of SCEE 2018.
	
	\bibitem{ElmkiesJoly97}
	A.~Elmkies and P.~Joly.
	\newblock \'el\'ements finis d'ar\^ete et condensation de masse pour les
	\'equations de {M}axwell: le cas de dimension {$3$}.
	\newblock {\em C. R. Acad. Sci. Paris S\'er. I Math.}, 325:1217--1222, 1997.
	
	\bibitem{GeeversMulderVegt18}
	S.~Geevers, W.~Mulder, and J.~van~der Vegt.
	\newblock New higher-order mass-lumped tetrahedral elements for wave
	propagation modelling.
	\newblock {\em SIAM Journal on Scientific Computing}, 40:A2830--A2857, 2018.
	
	\bibitem{HesthavenWarburton08}
	J.~S. Hesthaven and T.~Warburton.
	\newblock {\em Nodal discontinuous {G}alerkin methods}, volume~54 of {\em Texts
		in Applied Mathematics}.
	\newblock Springer, New York, 2008.
	
	\bibitem{Joly03}
	P.~Joly.
	\newblock Variational methods for time-dependent wave propagation problems.
	\newblock In {\em Topics in Computational Wave Propagation}, volume~31 of {\em
		LNCSE}, pages 201--264. Springer.
	
	\bibitem{Leis85}
	R.~Leis.
	\newblock {\em Initial Boundary Value Problems in Mathematical Physics}.
	\newblock Springer Fachmedien, Wiesbaden, 1985.
	
	\bibitem{MakridakisMonk95}
	C.~G. Makridakis and P.~Monk.
	\newblock Time-discrete finite element schemes for {M}axwell's equations.
	\newblock {\em RAIRO Model. Math. Anal. Numer.}, 29:171--197, 1995.
	
	\bibitem{Monk92}
	P.~Monk.
	\newblock Analysis of a finite element methods for {M}axwell's equations.
	\newblock {\em SIAM J. Numer. Anal.}, 29:714--729, 1992.
	
	\bibitem{Monk93}
	P.~Monk.
	\newblock An analysis of {N}édélec's method for the spatial discretization of
	{M}axwell's equations.
	\newblock {\em J. Comput. Appl. Math.}, 47:101--121, 1993.
	
	\bibitem{Pazy83}
	A.~Pazy.
	\newblock {\em Semigroups of linear operators and applications to partial
		differential equations}, volume~44 of {\em Applied Mathematical Sciences}.
	\newblock Springer-Verlag, New York, 1983.
	
	\bibitem{SchuhmannWeiland98}
	R.~Schuhmann and T.~Weiland.
	\newblock A stable interpolation technique for {FDTD} on non-orthogoanl grids.
	\newblock 11:299--306, 1998.
	
	\bibitem{Weiland77}
	T.~{Weiland}.
	\newblock {A discretization model for the solution of {M}axwell's equations for
		six-component fields}.
	\newblock {\em Archiv Elektronik und Uebertragungstechnik}, 31:116--120, 1977.
	
	\bibitem{Yee66}
	K.~Yee.
	\newblock Numerical solution of initial boundary value problems involving
	{M}axwell’s equations in isotropic media.
	\newblock {\em IEEE Trans. Antennas and Propagation}, AP-16:302--307, 1966.
	
\end{thebibliography}

\appendix

\section{Proof of Theorem~\ref{thm:1}} \label{sec:app}

In principle, the result of Theorem~\ref{thm:1} follows by standard arguments; see e.g. \cite{Dupont73,Joly03}. For convenience of the reader, we provide a detailed proof.
We denote by $\Pi_h$ the canonical projection operator for $\N_k(\Th)$ defined by $(\Pi_h u)|_\K = \Pi_\K^k u|_\K$; see assumption (A0).
\begin{lemma}\label{cor:estquad}
Let $\phi_h\in W^{1,\infty}([0,t],V_h)$ with $\phi_h(0)=0$ and let (A0)-(A2) hold. Then
\begin{align*}
\int_0^t\sigma_h(\Pi_h E(s),\dt \phi_h(s))\,ds &\leq C(E)^2 h^{2r} + \frac{1}{8}\|\dt \phi_h\|_{L^\infty(0,t,L^2(\Omega))}^2 + 
\frac{1}{8}\|\curl\,\phi_h\|_{L^\infty(0,t,L^2(\Omega))}^2
\end{align*}
with $C(E) = C\left(c_\sigma\|E\|_{L^\infty(0,t;H^{r-1}(\Th))} + \|E\|_{L^1(0,t;H^{r}(\Th))} +  c_\sigma\|\dt E\|_{L^1(0,t;H^{r-1}(\Th))}\right)$
and $r=\min\{k+1,q+1\}$ with constant $C$ depending only on the shape regularity of $\Th$.
\end{lemma}
\begin{proof}
The left hand side can be split into
\begin{align*}
\int_0^t\sigma_h(\Pi_h E,\dt \phi_h) &= \int_0^t\sigma_h(\Pi_h E-\pi_h^* E,\dt \phi_h) + \int_0^t\sigma_h(\pi_h^* E,\dt \phi_h) \\
&= \int_0^t\sigma_h(\Pi_h E-\pi_h^* E,\dt \phi_h) + \int_0^t\frac{d}{dt}\sigma_h(\pi_h^* E,\phi_h)-
\int_0^t\sigma_h(\pi_h^* \dt E,\phi_h)\\
&= (i) + (ii) + (iii).
\end{align*}
Here, $\pi_h^*$ is the projection introduced in assumption (A2). For the first term, we use the approximation properties resulting from assumptions (A0) and (A2) to estimate
\begin{align*}
(i)&\leq \int_0^t\|\Pi_h E-\pi_h^* E\|_{L^2(\Omega)}\|\dt \phi_h\|_{L^2(\Omega)}\\
&\leq \int_0^t\left(\|E-\pi_h^* E\|_{L^2(\Omega)} + \|E-\Pi_h E\|_{L^2(\Omega)}\right)\|\dt \phi_h\|_{L^2(\Omega)}\\
&\leq \int_0^t Ch^{r}\|E\|_{H^{{r}}(\Th)}\|\dt \phi_h\|_{L^2(\Omega)}\\
&\leq 2C'h^{2r}\|E\|_{L^1(0,t;H^{r}(\Th))}^2 + \tfrac{1}{8}\|\dt \phi_h\|_{L^\infty(0,t,L^2(\Omega))}^2.
\end{align*}
For the second term, we use $\phi_h(0)=0$ and assumption (A2), and obtain
\begin{align*}
(ii) =\sigma_h(\pi_h^* E(t),\phi_h(t))&\leq h^{q+1} c_\sigma\|E(t)\|_{H^q(\Th)}\|\curl\,\phi_h(t)\|_{L^2(\Omega)}\\
&\leq 4c_\sigma^2 h^{2q+2}\|E(t)\|_{H^q(\Th)}^2 + \tfrac{1}{16}\|\curl\,\phi_h(t)\|_{L^2(\Omega)}^2 \\
&\leq 4c_\sigma^2h^{2q+2}\|E\|_{L^\infty(0,t;H^{q}(\Th))}^2 + \tfrac{1}{16}\|\curl\,\phi_h\|_{L^\infty(0,t,L^2(\Omega))}^2.
\end{align*}
For the third term, using the same arguments as for $(ii)$, we write
\begin{align*}
(iii) &= \int_0^t h^{q+1} c_\sigma\|E(s)\|_{H^q(\Th)}\|\curl\,\phi_h(s)\|_{L^2(\Omega)}\,ds\\
&\leq 4c_\sigma^2h^{2q+2}\|\dt E\|_{L^1(0,t;H^q(\Th))}^2 + \tfrac{1}{16}\|\curl\,\phi_h(t)\|_{L^\infty(0,t,L^2(\Omega))}^2.
\end{align*}
Summing all the terms yields the result.
\end{proof}

\begin{proof}[Proof of Theorem \ref{thm:1}]
We split the error into discrete and projection error
\begin{align*}
E-E_h = -(\Pi_h E-E) + (\Pi_h E - E_h) \eqqcolon -\eta + \psi_h.
\end{align*}
The projection error $\eta$ can be bounded using the bounds \eqref{eq:projest1}--\eqref{eq:projest2} by
\begin{align*}
\|\dt\eta\|_{L^\infty(0,T,L^2(\Omega))} &+ \|\curl\,\eta\|_{L^\infty(0,T,L^2(\Omega))} \\
&\leq Ch^{k+1}(\|\dt E\|_{L^\infty(0,T;H^{k+1}(\Th))} + \|\curl\,E\|_{L^\infty(0,T;H^{k+1}(\Th))}).
\end{align*}
We now turn to the discrete error $\psi_h$. Due to the choice of initial values, we have $\psi_h(0)=0$ and $\dt\psi_h(0)=0$ and consequently, also $\curl\,\psi_h(0)=0$. Using \eqref{eq:var} and \eqref{eq:varh}, we obtain
\begin{align*}
(\dtt \psi_h,\phi_h)_h + (\curl\,\psi_h,\curl\,\phi_h)=(\dtt\eta,\phi_h) + (\curl\,\eta,\curl\,\phi_h) + \sigma_h(\Pi_h \dtt E,\phi_h).
\end{align*}
Choosing $\phi_h=\dt\psi_h(t)$ as test function and integrating from $0$ to $t$ further yields
\begin{align*}
&\frac{1}{2}\Big(\|\dt\psi_h(t)\|^2  + \|\curl\,\psi_h(t)\|_{L^2(\Omega)}^2\Big)\\
& = \int_0^t (\dtt \eta(s),\dt \psi_h(s)) + \int_0^t(\curl\,\eta(s),\curl\,\dt \psi_h(s)) + \int_0^t\sigma_h(\Pi_h \dtt E(s),\dt \psi_h(s)) \\
&= (i)+(ii)+(iii).
\end{align*}
The first term can be estimated via Cauchy-Schwarz and Young's inequality by
\begin{align*}
(i)&\leq \int_0^t\|\dtt\eta\|_{L^2(\Omega)} \|\dt\psi_h\|_{L^2(\Omega)}
\leq \|\dtt\eta\|^2_{L^1(0,t,L^2(\Omega))} + \frac{1}{4}\|\dt\psi_h\|^2_{L^\infty(0,t,L^2(\Omega))}\\
&\leq Ch^{2k+2}\|\dtt E\|^2_{L^1(0,t;H^{k+1}(\Th))} + \frac{1}{4}\|\dt\psi_h\|^2_{L^\infty(0,t,L^2(\Omega))}.
\end{align*}
For the second term, we apply integration by parts and obtain
\begin{align*}
(ii)&= \int_0^t \frac{d}{dt}(\curl\,\eta,\curl\,\psi_h) - \int_0^t (\curl\,\dt\eta,\curl\,\psi_h) = (ii_a) +(ii_b).
\end{align*}
For $(ii_a)$, since $\curl\,\psi_h(0)=0$, we have
\begin{align*}
(ii_a) &= (\curl\,\eta(t),\curl\,\psi_h(t))\leq 2\|\curl\,\eta(t)\|_{L^2(\Omega)}^2+\frac{1}{8}\|\curl\,\psi_h(t)\|_{L^2(\Omega)}^2\\
&\leq 2Ch^{2k+2}\|\curl\,E\|_{L^\infty(0,t;H^{k+1}(\Th))}^2+\frac{1}{8}\|\curl\,\psi_h\|_{L^\infty(0,t,L^2(\Omega))}^2,
\end{align*}
The term $(ii_b)$ is estimated again in a standard way by
\begin{align*}
(ii_b) &\leq 2Ch^{2k+2}\|\curl\,\dt E\|^2_{L^1(0,t;H^{k+1}(\Th))} + \frac{1}{8}\|\curl\,\psi_h\|^2_{L^\infty(0,t,L^2(\Omega))}.
\end{align*}
Using Corollary~\ref{cor:estquad}, we can estimate the third term by
\begin{align*}
(iii)&\leq C(\dtt E)^2 h^{2r} + \frac{1}{8}\|\dt\psi_h\|_{L^\infty(0,t,L^2(\Omega))}^2 + \frac{1}{8}\|\curl\,\psi_h\|_{L^\infty(0,t,L^2(\Omega))}^2,
\end{align*}
with 
\begin{align*}
C(\dtt E) = C \left(c_\sigma\|\dtt E\|_{L^\infty(0,t;H^{r-1}(\Th))}^2 + \|\dtt E\|_{L^1(0,t;H^{r}(\Th))}^2 + 
c_\sigma\|\dttt E\|_{L^1(0,t;H^{r-1}(\Th))}^2\right).
\end{align*}
Summing all the terms and using assumption (A1), we obtain
\begin{align*}
&\|\dt\psi_h(t)\|_{L^2(\Omega)}^2  + \|\curl\,\psi_h(t)\|_{L^2(\Omega)}^2\\
&\leq \widetilde C(E)^2 h^{2r} +\frac{1}{2}\|\dt\psi_h\|_{L^\infty(0,t,L^2(\Omega))}^2 + \frac{1}{2}\|\curl\,\psi_h\|_{L^\infty(0,t,L^2(\Omega))}^2, \numberthis \label{eq:beweis3}
\end{align*}
with 
\begin{align*}
\widetilde C(E)^2 = C(\dtt E)^2 + \; & C\Big(\|\curl\,E\|_{L^\infty(0,t;H^{r}(\Th))}^2 + \|\curl\,\dt E\|^2_{L^1(0,t;H^{r}(\Th))}\Big).
\end{align*}
Taking the maximum over all $t$ in \eqref{eq:beweis3} and subsequently absorbing the last two terms by the left hand side yields the $L^\infty$-estimate. The main result follows by adding the two results for the interpolation and discrete error component.
\end{proof}

\section{Proof of Theorem~\ref{thm:2}} \label{sec:app2}

We now continue with the error estimate for the fully discrete method. Let
\begin{align*}
\widehat{\dtau a_h^{\,n}}\coloneqq 
\frac{\dtau a_h^{\,n+\frac{1}{2}}+\dtau a_h^{\,n-\frac{1}{2}}}{2}=
\frac{a_h^{n+1}-a_h^{n-1}}{2\tau}.
\end{align*}
We then get the following estimate for the discrete energy.
\begin{lemma}\label{lem:discrete}
Let $\{a_h^{n}\},\{f_h^{n}\} \subset V_h$ be given sequences with $a_h^1=a_h^0=0$ such that 
\begin{alignat}{2}
(\dtautau a_h^{n},v_h)_{h} + (\curl\,a_h^{n}, \curl \, v_h) &= (f_h^n,v_h) \qquad &&\forall v_h \in V_h, \label{eq:dee1}
\end{alignat}
Furthermore, assume that (A1) and (A3) hold. Then for all $0 \le n \le N-1$, we have
\begin{align}
\|\dtau a_h^{n+\frac{1}{2}}\|_{L^2(\Omega)}^2 + \|\curl\,\widehat a_h^{\,n+\frac{1}{2}}\|_{L^2(\Omega)}^2
& \le C\sum_{i=1}^{n} \tau (f_h^i,\widehat{\dtau a_h^{\,n}}). \label{eq:discenergyest}
\end{align}
\end{lemma}
\begin{proof}
We follow the arguments of \cite{Joly03}.
Testing \eqref{eq:dee1} with $v_h=\widehat{\dtau a_h^n}$ leads to
\begin{align}
(f_h^n,\widehat{\dtau a_h^{\,n}})=(\dtautau a_h^{n},\widehat{\dtau a_h^{\,n}})_{h}+(\curl\,a_h^{n},
\curl\,\widehat{\dtau a_h^{\,n}})=(i)+(ii).\label{eq:al1}
\end{align}
For the first term, we obtain
\begin{align*}
(i) = \frac{1}{2\tau}\left(\|\dtau a_h^{n+\frac{1}{2}}\|_{h,\Omega}^2- \|\dtau a_h^{n-\frac{1}{2}}\|_{h,\Omega}^2\right),
\end{align*}
and the second term can be expanded as 
\begin{align*}
(ii) = \frac{1}{\tau}\Big(\|\curl\,\widehat a_h^{\,n+\frac{1}{2}}\|_{L^2(\Omega)}^2&-
\|\curl\,\widehat a_h^{\,n-\frac{1}{2}}\|_{L^2(\Omega)}^2\\
&-\frac{\tau^2}{4}\|\curl\,\dtau a_h^{n+\frac{1}{2}}\|_{L^2(\Omega)}^2
+\frac{\tau^2}{4}\|\curl\,\dtau a_h^{n-\frac{1}{2}}\|_{L^2(\Omega)}^2\Big).
\end{align*}
We now define the discrete energy as 
\begin{align*}
\E_h^{n+\frac{1}{2}} = \frac{1}{2}\|\dtau a_h^{n+\frac{1}{2}}\|_{h,\Omega}^2
+\|\curl\,\widehat a_h^{\,n+\frac{1}{2}}\|_{L^2(\Omega)}^2
-\frac{\tau^2}{4}\|\curl\,\dtau a_h^{n+\frac{1}{2}}\|_{L^2(\Omega)}^2.
\end{align*}
Plugging this definition back into \eqref{eq:al1} yields
\begin{align}
\E_h^{n+\frac{1}{2}}&= \E_h^{n-\frac{1}{2}}+\tau(f_h^n,\widehat{\dtau a_h^{\,n}}) \label{eq:difp}.
\end{align}
A recursive application of this inequality finally leads to 
$
\E_h^{n+\frac{1}{2}} = \E_h^{\frac{1}{2}} + \sum_{i=1}^{n} \tau(f_h^i,\widehat{\dtau a_h^{\,i}}).
$
From (A3), we have that $\E_h^{n+\frac{1}{2}}$ is positive and
\begin{align}
\E_h^{n+\frac{1}{2}}\leq  \|\dtau a_h^{n+\frac{1}{2}}\|_{h,\Omega}^2+ 
\|\curl\,\widehat a_h^{\,n+\frac{1}{2}}\|_{L^2(\Omega)}^2
\leq 2\E_h^{n+\frac{1}{2}}.
\end{align}
Moreover, $\E^{\frac{1}{2}}=0$ since $a_h^1=a_h^0=0$. The assertion of the Lemma now follows by using the norm equivalence estimate in assumption (A1).
\end{proof}

\begin{proof}[Proof of Theorem \ref{thm:2}]
We proceed as in the proof of Theorem~\ref{thm:1}. We again split the error in discrete and projection error components
\begin{align*}
E(t^n)-E_h^n = -(\Pi_h E(t^n)-E(t^n)) + (\Pi_h E(t^n) - E_h^n) \eqqcolon -b_h^n+a_h^n.
\end{align*}
For our choice of initial values, see definition in Problem~\ref{prob:full}, we have $a_h^1=a_h^0=0$. 
The projection error can be estimated as in the proof of Theorem~\ref{thm:1}. For the discrete error, we now use Lemma~\ref{lem:discrete} and estimate the right hand side of \eqref{eq:discenergyest} to obtain
\begin{align*}
\sum_{i=1}^{n}\tau(f_h^i,\widehat{\dtau a_h^{\,i}})
&=\tau\sum_{i=1}^{n}(\dtautau\Pi_h E(t^i)-\dtt\Pi_h E(t^i),\widehat{\dtau a_h^{\,i}})_h+
\tau\sum_{i=1}^{n}(\dtt b_h^i,\widehat{\dtau a_h^{\,i}})\\
&+\tau\sum_{i=1}^{n}(\curl\,b_h^i,\curl\,\widehat{\dtau a_h^{\,i}})-
\tau\sum_{i=1}^{n}\sigma_h(\dtt\Pi_h E(t^i),\widehat{\dtau a_h^{\,i}})\\
&=(i)+(ii)+(iii)+(iv).
\end{align*}
The first term can be estimated by Taylor expansions, assumption (A1), and Cauchy-Schwarz inequalities, yielding
\begin{align*}
|(i)|&\leq \tau\sum_{i=1}^{n}C\tau^2\|\dtttt \Pi_h E\|_{L^1(t^{i-1},t^{i+1},L^2(\Omega))}
\|\widehat{\dtau a_h^{\,i}}\|_{L^2(\Omega)}\\
&\leq \sum_{i=1}^{n}C\tau^2\|\dtttt E\|_{L^1(t^{i-1},t^{i+1};H^1(\Th))}\cdot 
\max\limits_{0\le i\le n}\|\widehat{\dtau a_h^{\,i}}\|_{L^2(\Omega)}\\
&\leq 2C\tau^{4}\|\dtttt E\|_{L^1(0,t^{n+1};H^1(\Th))}^2
+\frac{1}{8}\max\limits_{0\le i\le n}\|\dtau a_h^{i+\frac{1}{2}}\|_{L^2(\Omega)}^2.
\end{align*}
In a similar fashion, we get for the second term
\begin{align*}
|(ii)|&\leq 2Ch^{2k+2}\|\dtt E\|_{L^2(0,t^{n+1};H^{k+1}(\Th))}^2
+\frac{1}{8}\max\limits_{0\le i\le n}\|\dtau a_h^{i+\frac{1}{2}}\|_{L^2(\Omega)}^2.
\end{align*}
For the third term, we use summation by parts and $\curl\,a_h^{1}=0$ to obtain
\begin{align*}
(iii)&= (\curl\,b_h^{n+1},\curl\,a_h^{n})-(\curl\,b_h^{n},\curl\,a_h^{n+1})+
\tau\sum_{i=2}^{n}(\curl\,\widehat{\dtau b_h^{i}},\curl\,a_h^{i})\\
&=(iii)_a+(iii)_b+(iii)_c.
\end{align*}
For first two terms can be estimated by
\begin{align*}
|(iii)_a|+|(iii)_b|
&\leq 2Ch^{2k+2}\|\curl\,E\|_{L^\infty(0,t^{n+1};H^{k+1}(\Th))}^2+
\frac{1}{8}\max\limits_{0\le i\le n}\|\curl\,a_h^i\|_{L^2(\Omega)}^2.
\end{align*}
For the third term, we obtain
\begin{align*}
|(iii)_c|&\le 2Ch^{2k+2}\|\curl\,\dt E\|_{L^1(0,t^{n+1};H^1(\Th))}^2
+\frac{1}{8}\max\limits_{0\le i\le n}\|\curl\, a_h^{i+\frac{1}{2}}\|_{L^2(\Omega)}^2.
\end{align*}
Using a discrete variant of Lemma~\ref{cor:estquad}, we get 
\begin{align*}
|(iv)| &\leq C(E)^2 h^{2r} + 
\frac{1}{8}\|\dtau a_h^n\|_{L^\infty(0,t^{n+1},L^2(\Omega))}^2 + 
\frac{1}{8}\|\curl\,a_h^n\|_{L^\infty(0,t^{n+1},L^2(\Omega))}^2,
\end{align*}
with
\begin{align*}
C(E) = C \left(c_\sigma\|E\|_{L^\infty(0,t^{n+1};H^{r-1}(\Th))} + \|E\|_{L^1(0,t^{n+1};H^{r}(\Th))} + 
c_\sigma\|\dt E\|_{L^1(0,t^{n+1};H^{r-1}(\Th))}\right).
\end{align*}
Summing all the terms and using assumption (A1), we obtain
\begin{align}
&\|\dtau a_h^{n+\frac{1}{2}}\|_{L^2(\Omega)}^2 + \|\curl\,\widehat a_h^{\,n+\frac{1}{2}}\|_{L^2(\Omega)}^2\nonumber\\
&\leq \widetilde C(E)^2 h^{2k+2}+\widehat C(E)^2 \tau^4 +\tfrac{1}{2}\max\limits_{0\le i\le n}
\Big(\|\dtau a_h^{i+\frac{1}{2}}\|_{L^2(\Omega)}^2 + \|\curl\,\widehat a_h^{\,i+\frac{1}{2}}\|_{L^2(\Omega)}^2\Big),
\label{eq:beweis4}
\end{align}
with 
$
\widetilde C(E)^2  = C(\dtt E)^2 + C\Big(\|\curl\,E\|_{L^\infty(0,t^{n+1};H^{r}(\Th))}^2 + 
\|\curl\,\dt E\|^2_{L^1(0,t^{n+1};H^{r}(\Th))}\Big)
$
and $\widehat C(E)=C\|\dtttt E\|_{L^1(0,t^{n+1};H^1(\Th))}$. 
Taking the maximum over all $t^n$ in \eqref{eq:beweis4} and absorbing the last two terms by the left hand side yields the estimate. The assertion of the theorem now follows by adding the two estimates for the projection and discrete error.
\end{proof}

\end{document}